\documentclass[10pt,a4paper]{amsart}
\usepackage{eurosym}
\usepackage{graphicx,amscd,color,amsmath,amsfonts,amssymb,geometry,hyperref,soul}

\newtheorem{theorem}{Theorem}
\theoremstyle{plain}

\newtheorem{lemma}{Lemma}

\newtheorem{remark}{Remark}

\numberwithin{equation}{section}
\begin{document}
\title[Hardy--Littlewood constants]{A new estimate for the constants of an
inequality due to Hardy and Littlewood}
\author[Antonio Gomes Nunes]{Antonio Gomes Nunes}
\address[A.G. Nunes]{Department of Mathematics - CCEN \\
\indent UFERSA \\
\indent Mossor\'{o}, RN, Brazil.}
\email{nunesag@gmail.com and nunesag@ufersa.edu.br}
\thanks{Partially supported by Capes.}

\begin{abstract}
In this paper we provide a family of inequalities, extending a recent result
due to Albuquerque \textit{et al.}
\end{abstract}

\maketitle

\section{Introduction}

The Hardy--Littlewood inequalities \cite{hl} for $m$--linear forms and
polynomials (see \cite{a22, a, a11, ap2, dimant, pt, pra}) are perfect
extensions of the Bohnenblust--Hille inequality \cite{bh} when the sequence
space $c_{0}$ is replaced by the sequence space $\ell _{p}$. These
inequalities assert that for any integer $m\geq 2$ there exist constants $%
C_{m,p}^{\mathbb{K}},D_{m,p}^{\mathbb{K}}\geq 1$ such that
\begin{equation}
\left( \sum_{j_{1},\cdots ,j_{m}=1}^{\infty }\left\vert T(e_{j_{1}},\cdots
,e_{j_{m}})\right\vert ^{\frac{2mp}{mp+p-2m}}\right) ^{\frac{mp+p-2m}{2mp}%
}\leq C_{m,p}^{\mathbb{K}}\left\Vert T\right\Vert ,  \label{i99}
\end{equation}%
when $2m\leq p\leq \infty $, and%
\begin{equation*}
\left( \sum_{j_{1},\cdots ,j_{m}=1}^{\infty }\left\vert T(e_{j_{1}},\cdots
,e_{j_{m}})\right\vert ^{\frac{p}{p-m}}\right) ^{\frac{p-m}{p}}\leq D_{m,p}^{%
\mathbb{K}}\left\Vert T\right\Vert ,
\end{equation*}%
when $m<p\leq 2m$, for all continuous $m$--linear forms $T:\ell _{p}\times
\cdots \times \ell _{p}\rightarrow \mathbb{K}$ (here, and henceforth, $%
\mathbb{K}=\mathbb{R}$ or $\mathbb{C}$). Both exponents are optimal, $i.e.,$
cannot be smaller without paying the price of a dependence on $n$ arising on
the respective constants. Following usual convention in the field, $c_{0}$
is understood as the substitute of $\ell _{\infty }$ when the exponent $p$
goes to infinity.

The investigation of the optimal constants of the Hardy--Littlewood
inequalities is closely related to the fashionable, mysterious and puzzling
investigation of the optimal Bohnenblust--Hille inequality constants (see,
for instance \cite{pt} and the references therein).

In this note we extend the following result of \cite[Theorem 3]{rrrr}:

\begin{theorem}[Albuquerque et al.]
Let $m\geq 2$ be a positive integer and $m<p\leq 2m-2.$ Then, for all
continuous $m$-linear forms $T:\ell _{p}\times \cdots \times \ell
_{p}\rightarrow \mathbb{K}$, we have
\begin{equation*}
\left( \sum_{j_{i}=1}^{n}\left( \sum_{\widehat{j_{i}}=1}^{n}\left\vert
T\left( e_{j_{1}},\ldots ,e_{j_{m}}\right) \right\vert ^{\frac{p}{p-\left(
m-1\right) }}\right) ^{\frac{p-\left( m-1\right) }{p}\cdot \frac{p}{p-m}%
}\right) ^{\frac{p-m}{p}}\leq 2^{\frac{\left( m-1\right) \left( p-m+1\right)
}{p}}\left\Vert T\right\Vert .
\end{equation*}
\end{theorem}

More precisely, using a different technique we find a family of inequalities
extending the above result. Our result reads as follows, where $A_{\lambda
_{0}}$ is the optimal constant of the Khinchin inequality (defined in
Section \ref{kkmm}):

\begin{theorem}
\label{kkmm}If $\lambda _{0}\in \lbrack 1,2]$ and%
\begin{equation*}
\lambda _{0}m<p\leq \frac{2\lambda _{0}\left( m-1\right) }{2-\lambda _{0}},
\end{equation*}%
then
\begin{equation*}
\left( \sum_{j_{i}=1}^{n}\left( \sum_{\widehat{j_{i}}=1}^{n}\left\vert
T\left( e_{j_{1}},\dots ,e_{j_{m}}\right) \right\vert ^{s}\right) ^{\frac{1}{%
s}\eta _{1}}\right) ^{\frac{1}{\eta _{1}}}\leq A_{\lambda _{0}}^{\frac{%
-2\left( m-1\right) }{s}}\left\Vert T\right\Vert
\end{equation*}
for
\begin{eqnarray*}
\eta _{1} &=&\frac{\lambda _{0}p}{p-\lambda _{0}m}, \\
s &=&\frac{\lambda _{0}p}{p-\lambda _{0}m+\lambda _{0}},
\end{eqnarray*}%
all $m$-linear forms $T:\ell _{p}^{n}\times \cdots \times \ell
_{p}^{n}\rightarrow \mathbb{K}$, and all positive integers $n.$
\end{theorem}

Our main technique is based on an original argument developed in \cite{rrrr}%
, with some slight technical changes.

\section{The proof of Theorem \protect\ref{kkmm}}

Let $m\geq 2$ be a positive integer, $F$ be a Banach space, $A\subset
I_{m}:=\{1,\dots ,m\},p_{1},\dots ,p_{m},s,\alpha \geq 1$ and
\begin{equation*}
B_{p_{1},\dots ,p_{m}}^{A,s,\alpha ,F,n}:=\inf \left\{ C(n)\,:\left(
\sum_{j_{i}=1}^{n}\left( \sum_{\widehat{j_{i}}=1}^{n}\left\Vert T\left(
e_{j_{1}},\dots ,e_{j_{m}}\right) \right\Vert ^{s}\right) ^{\frac{1}{s}%
\alpha }\right) ^{\frac{1}{\alpha }}\leq C(n),\text{ for all }i\in A\right\}
,
\end{equation*}%
in which $\widehat{j_{i}}$ means that the sum runs over all indexes but $%
j_{i}$, and the infimum is taken over all norm-one $m$-linear operators $%
T:\ell _{p_{1}}^{n}\times \cdots \times \ell _{p_{m}}^{n}\rightarrow F$. We
begin by recalling the following lemma proved in \cite{rrrr}:

\begin{lemma}
\label{geral} Let $1\leq p_{k}<q_{k}\leq \infty ,\,k=1,\dots ,m$ and $%
\lambda _{0},s\geq 1$. If%
\begin{equation}
\sum_{j=1}^{m}\left( \frac{1}{p_{j}}-\frac{1}{q_{j}}\right) <\frac{1}{%
\lambda _{0}}\quad \text{ and }\quad s\geq \left[ \frac{1}{\lambda _{0}}%
-\sum_{j=1}^{m-1}\left( \frac{1}{p_{j}}-\frac{1}{q_{j}}\right) \right]
^{-1}=:\eta _{2}  \label{hipinica}
\end{equation}%
then
\begin{equation*}
B_{p_{1},\dots ,p_{m}}^{\{m\},s,\eta _{1},F,n}\leq B_{q_{1},\dots
,q_{m}}^{I_{m},s,\lambda _{0},F,n},
\end{equation*}%
where%
\begin{equation*}
\eta _{1}:=\left[ \frac{1}{\lambda _{0}}-\sum_{j=1}^{m}\left( \frac{1}{p_{j}}%
-\frac{1}{q_{j}}\right) \right] ^{-1}
\end{equation*}
\end{lemma}

We also need to recall the Khinchin inequality: for any $0<q<\infty $, there
are positive constants $A_{q}$, $B_{q}$ such that regardless of the positive
integer $n$ and of the scalar sequence $(a_{j})_{j=1}^{n}$ we have
\begin{equation*}
A_{q}\left( \sum\limits_{j=1}^{n}|a_{j}|^{2}\right) ^{\frac{1}{2}}\leq
\left( \int\nolimits_{\lbrack 0,1]}\left\vert
\sum\limits_{j=1}^{n}a_{j}r_{j}(t)\right\vert ^{q}dt\right) ^{\frac{1}{q}%
}\leq B_{q}\left( \sum\limits_{j=1}^{n}|a_{j}|^{2}\right) ^{\frac{1}{2}},
\end{equation*}%
where $r_{j}$ are the classical Rademacher functions (random variables).

The best constants $A_{q}$ are the following ones (see \cite{haage}):

\begin{itemize}
\item $\displaystyle A_{q}=\sqrt{2}\left( \frac{\Gamma \left( \frac{1+q}{2}%
\right) }{\sqrt{\pi }}\right) ^{\frac{1}{q}}$ if $q>q_{0}\cong 1.85$;

\item $\displaystyle A_{q}=2^{\frac{1}{2}-\frac{1}{q}}$ if $q<q_{0},$
\end{itemize}

where $q_{0}\in (1,2)$ is the only real number such that $\Gamma \left(
\frac{p_{0}+1}{2}\right) =\frac{\sqrt{\pi }}{2}$. For complex scalars, using
Steinhaus variables instead of Rademacher functions it is well known that a
similar inequality holds, but with better constants. In this case the
optimal constant is

\begin{itemize}
\item $\displaystyle A_{q}=\Gamma\left( \frac{q+2}{2}\right) ^{\frac{1}{q}}$
if $q\in\lbrack1,2]$.
\end{itemize}

The notation of the constants $A_{q}$ above will be used in all this paper.
The following result is a variant of \cite{rrrr}, and is based on the
Contraction Principle (see \cite[Theorem 12.2]{Di}). From now on $r_{i}(t)$
are the Rademacher functions.

\begin{lemma}
\label{778}Regardless of the choice of the positive integers $m,N$ and the
scalars $a_{i_{1},\dots ,i_{m}}$, $i_{1},\dots ,i_{m}=1,\dots ,N$,
\begin{equation*}
\max_{\substack{ i_{k}=1,\dots ,N  \\ k=1,\dots ,m}}\left\vert
a_{i_{1},\dots ,i_{m}}\right\vert \leq \left( \int_{\lbrack
0,1]^{m}}\left\vert \sum_{i_{1},\dots ,i_{m}=1}^{N}r_{i_{1}}(t_{1})\cdots
r_{i_{m}}(t_{m})a_{i_{1},\dots ,i_{m}}\right\vert ^{t}\,dt_{1}\cdots
dt_{m}\right) ^{1/t}
\end{equation*}
for all $t\geq1$.
\end{lemma}

\begin{proof}
The proof is an adaptation of an argument used in \cite{rrrr}. Essentially,
we have to use the Contraction Principle inductively. The case $m=1$ is
nothing else than the standard version of Contraction Principle (see \cite[%
Theorem 12.2]{Di}). For all positive integers $i_{1},\dots ,i_{m}$,
\begin{align*}
& \left( \int_{[0,1]^{m}}\left\vert \sum_{i_{1},\dots
,i_{m}=1}^{N}r_{i_{1}}(t_{1})\cdots r_{i_{m}}(t_{m})a_{i_{1},\dots
,i_{m}}\right\vert ^{t}\,dt_{1}\cdots dt_{m}\right) ^{1/t} \\
& =\left( \int_{[0,1]^{m-1}}\left[ \left( \int_{0}^{1}\left\vert
\sum_{i_{1}=1}^{N}r_{i_{1}}(t_{1})\left( \sum_{i_{2},\dots
,i_{m}=1}^{N}r_{i_{2}}(t_{2})\cdots r_{i_{m}}(t_{m})a_{i_{1},\dots
,i_{m}}\right) \right\vert ^{t}\,dt_{1}\right) ^{\frac{1}{t}}\right]
^{t}\,dt_{2}\cdots dt_{m}\right) ^{1/t} \\
& \geq \left( \int_{\lbrack 0,1]^{m-1}}\left\vert \sum_{i_{2},\dots
,i_{m}=1}^{N}r_{i_{2}}(t_{2})\cdots r_{i_{m}}(t_{m})a_{i_{1},\dots
,i_{m}}\right\vert ^{t}\,dt_{2}\cdots dt_{m}\right) ^{1/t} \\
& \geq \left\vert a_{i_{1},\dots ,i_{m}}\right\vert ,
\end{align*}%
where we used the Contraction Principle and the induction hypothesis on the
first and second inequalities, respectively. This concludes the proof of the
lemma.
\end{proof}

Now we are able to complete the proof. Let $S:\ell _{\infty }^{n}\times
\cdots \times \ell _{\infty }^{n}\rightarrow \mathbb{K}$ be an $m$-linear
form and consider
\begin{equation*}
s=\left[ \frac{1}{\lambda _{0}}-\sum_{j=1}^{m-1}\left( \frac{1}{p_{j}}-\frac{%
1}{q_{j}}\right) \right] ^{-1}\geq 2
\end{equation*}%
and
\begin{equation*}
\lambda _{0}\in \lbrack 1,2].
\end{equation*}

Since $s\geq 2$, from Lemma \ref{778}, H\"{o}lder's inequality and
Khinchin's inequality for multiple sums (\cite{popa}), choosing $\theta =2/s$
we obtain%
\begin{align*}
& \left( \sum_{j_{1}=1}^{n}\left( \sum_{\widehat{j_{1}}=1}^{n}\left\vert
S\left( e_{j_{1}},\ldots ,e_{j_{m}}\right) \right\vert ^{s}\right) ^{\frac{1%
}{s}\lambda _{0}}\right) ^{\frac{1}{\lambda _{0}}} \\
& \leq \left( \sum_{j_{1}=1}^{n}\left( \left( \left( \sum_{\widehat{j_{1}}%
=1}^{n}\left\vert S\left( e_{j_{1}},\ldots ,e_{j_{m}}\right) \right\vert
^{2}\right) ^{\frac{1}{2}}\right) ^{\theta }\left( \max_{\widehat{j_{1}}%
}\left\vert S\left( e_{j_{1}},\ldots ,e_{j_{m}}\right) \right\vert \right)
^{1-\theta }\right) ^{\lambda _{0}}\right) ^{\frac{1}{\lambda _{0}}} \\
& \leq \left( \sum_{j_{1}=1}^{n}\left( \left( A_{\lambda _{0}}^{-\left(
m-1\right) }R_{n}\right) ^{\frac{2}{s}}R_{n}^{1-\frac{2}{s}}\right)
^{\lambda _{0}}\right) ^{\frac{1}{\lambda _{0}}} \\
& =A_{\lambda _{0}}^{\frac{-2\left( m-1\right) }{s}}\left(
\sum_{j_{1}=1}^{n}\int_{[0,1]^{m-1}}\left\vert \sum_{\widehat{j_{1}}%
=1}^{n}r_{j_{2}}(t_{2})\cdots r_{j_{m}}(t_{m})S\left( e_{j_{1}},\ldots
,e_{j_{m}}\right) \right\vert ^{\lambda _{0}}\,dt_{2}\cdots dt_{m}\right)
^{1/\lambda _{0}} \\
& =A_{\lambda _{0}}^{\frac{-2\left( m-1\right) }{s}}\left(
\int_{[0,1]^{m-1}}\sum_{j_{1}=1}^{n}\left\vert S\left(
e_{j_{1}},\sum_{j_{2}=1}^{n}r_{j_{2}}(t_{2})e_{j_{2}},\ldots
,\sum_{j_{m}=1}^{n}r_{j_{m}}(t_{m})e_{j_{m}}\right) \right\vert ^{\lambda
_{0}}dt_{2}\cdots dt_{m}\right) ^{1/\lambda _{0}} \\
& \leq A_{\lambda _{0}}^{\frac{-2\left( m-1\right) }{s}}\left(
\sup_{t_{2},\ldots ,t_{m}\in \lbrack 0,1]}\sum_{j_{1}=1}^{n}\left\vert
S\left( e_{j_{1}},\sum_{j_{2}=1}^{n}r_{j_{2}}(t_{2})e_{j_{2}},\ldots
,\sum_{j_{m}=1}^{n}r_{j_{m}}(t_{m})e_{j_{m}}\right) \right\vert ^{\lambda
_{0}}\right) ^{1/\lambda _{0}} \\
& \leq A_{\lambda _{0}}^{\frac{-2\left( m-1\right) }{s}}\left\Vert
S\right\Vert ,
\end{align*}%
where
\begin{equation*}
R_{n}:=\left( \int_{[0,1]^{m-1}}\left\vert \sum_{\widehat{j_{1}}%
=1}^{n}r_{j_{2}}(t_{2})\cdots r_{j_{m}}(t_{m})S\left( e_{j_{1}},\ldots
,e_{j_{m}}\right) \right\vert ^{\lambda _{0}}\,dt_{2}\cdots dt_{m}\right) ^{%
\frac{1}{\lambda _{0}}}.
\end{equation*}%
Repeating the same procedure for the other indexes we have%
\begin{equation*}
\left( \sum_{j_{i}=1}^{n}\left( \sum_{\widehat{j_{i}}=1}^{n}\left\vert
S\left( e_{j_{1}},\ldots ,e_{j_{m}}\right) \right\vert ^{s}\right) ^{\frac{1%
}{s}\lambda _{0}}\right) ^{\frac{1}{\lambda _{0}}}\leq A_{\lambda _{0}}^{%
\frac{-2\left( m-1\right) }{s}}\left\Vert S\right\Vert
\end{equation*}%
for all $i=1,...,m.$ Hence, from Lemma \ref{geral}, we conclude that
\begin{equation*}
\left( \sum_{j_{i}=1}^{n}\left( \sum_{\widehat{j_{i}}=1}^{n}\left\vert
T\left( e_{j_{1}},\dots ,e_{j_{m}}\right) \right\vert ^{s}\right) ^{\frac{1}{%
s}\eta _{1}}\right) ^{\frac{1}{\eta _{1}}}\leq A_{\lambda _{0}}^{\frac{%
-2\left( m-1\right) }{s}}\left\Vert T\right\Vert
\end{equation*}%
for all $m$-linear forms $T:\ell _{p}^{n}\times \cdots \times \ell
_{p}^{n}\rightarrow \mathbb{K}$ and all positive integers $n,$ where%
\begin{equation*}
\eta _{1}:=\left[ \frac{1}{\lambda _{0}}-\sum_{j=1}^{m}\left( \frac{1}{p_{j}}%
-\frac{1}{q_{j}}\right) \right] ^{-1}.
\end{equation*}%
We thus conclude that if
\begin{equation*}
s=\left[ \frac{1}{\lambda _{0}}-\frac{m-1}{p}\right] ^{-1}\geq 2
\end{equation*}%
and
\begin{equation*}
\lambda _{0}\in \lbrack 1,2]
\end{equation*}%
with%
\begin{equation*}
p>\lambda _{0}m
\end{equation*}%
and%
\begin{equation*}
\eta _{1}=\frac{\lambda _{0}p}{p-\lambda _{0}m},
\end{equation*}%
then
\begin{equation*}
\left( \sum_{j_{i}=1}^{n}\left( \sum_{\widehat{j_{i}}=1}^{n}\left\vert
T\left( e_{j_{1}},\dots ,e_{j_{m}}\right) \right\vert ^{s}\right) ^{\frac{1}{%
s}\eta _{1}}\right) ^{\frac{1}{\eta _{1}}}\leq A_{\lambda _{0}}^{\frac{%
-2\left( m-1\right) }{s}}\left\Vert T\right\Vert
\end{equation*}%
for all $m$-linear forms $T:\ell _{p}^{n}\times \cdots \times \ell
_{p}^{n}\rightarrow \mathbb{K}$ and all positive integers $n.$ In other
words, if
\begin{equation*}
\lambda _{0}\in \lbrack 1,2]
\end{equation*}%
with%
\begin{equation*}
\lambda _{0}m<p\leq \frac{2\lambda _{0}\left( m-1\right) }{2-\lambda _{0}}
\end{equation*}%
and%
\begin{eqnarray*}
\eta _{1} &=&\frac{\lambda _{0}p}{p-\lambda _{0}m}, \\
s &=&\frac{\lambda _{0}p}{p-\lambda _{0}m+\lambda _{0}},
\end{eqnarray*}%
then
\begin{equation*}
\left( \sum_{j_{i}=1}^{n}\left( \sum_{\widehat{j_{i}}=1}^{n}\left\vert
T\left( e_{j_{1}},\dots ,e_{j_{m}}\right) \right\vert ^{s}\right) ^{\frac{1}{%
s}\eta _{1}}\right) ^{\frac{1}{\eta _{1}}}\leq A_{\lambda _{0}}^{\frac{%
-2\left( m-1\right) }{s}}\left\Vert T\right\Vert
\end{equation*}%
for all $m$-linear forms $T:\ell _{p}^{n}\times \cdots \times \ell
_{p}^{n}\rightarrow \mathbb{K}$ and all positive integers $n.$

\bigskip

\begin{remark}
If $\lambda _{0}=1$ then we get%
\begin{equation*}
\left( \sum_{j_{i}=1}^{n}\left( \sum_{\widehat{j_{i}}=1}^{n}\left\vert
T\left( e_{j_{1}},\dots ,e_{j_{m}}\right) \right\vert ^{s}\right) ^{\frac{1}{%
s}\eta _{1}}\right) ^{\frac{1}{\eta _{1}}}\leq A_{1}^{\frac{-2\left(
m-1\right) }{s}}\left\Vert T\right\Vert
\end{equation*}%
with $m<p\leq 2m-2$ and
\begin{eqnarray*}
s &=&\frac{p}{p-m+1}, \\
\eta _{1} &=&\frac{p}{p-m}
\end{eqnarray*}%
and we recover \cite[Theorem 3]{rrrr}.
\end{remark}

\end{document}